\documentclass[11pt,a4paper]{article}

\usepackage[utf8]{inputenc}
\usepackage[T1]{fontenc}
\usepackage{amsmath,amsthm,amsopn,amsfonts,amssymb,dsfont,color}
\usepackage{mathrsfs}
\usepackage{tikz}
\usepackage[T1]{fontenc}
\usepackage{lmodern}

\usepackage{xcolor}

\usepackage{hyperref}

\usepackage{a4,a4wide}

\def\ep{{\varepsilon}}
\let\epsilon=\varepsilon
\def\R{\mathbb R}

\let\phi=\varphi

\newtheorem{theorem}{\textbf{Theorem}}[section]

\newtheorem{lemma}[theorem]{\textbf{Lemma}}
\newtheorem{proposition}[theorem]{\textbf{Proposition}}

\newtheorem{definition}[theorem]{\textbf{Definition}}

\theoremstyle{remark}
\newtheorem{remark}[theorem]{\textbf{Remark}}

\numberwithin{equation}{section}

\title{Lotka-Volterra competition-diffusion system: the critical competition case}
\date{}

\begin{document}

\maketitle

\begin{center}
{\large\bf  Matthieu Alfaro\footnote{Universit\'e de Rouen Normandie, CNRS, Laboratoire de Math\'ematiques Rapha\"el Salem, Saint-Etienne-du-Rouvray, France \& BioSP, INRAE, 84914, Avignon, France. e-mail: {\tt matthieu.alfaro@univ-rouen.fr}} and Dongyuan Xiao\footnote{IMAG, Univ. Montpellier, CNRS, Montpellier, France. e-mail: {\tt dongyuan.xiao@umontpellier.fr}}}\\
[2ex]
\end{center}


\tableofcontents

\vspace{10pt}

\begin{abstract} 
We consider the reaction-diffusion competition system in the so-called {\it critical competition case}. The associated ODE system then admits infinitely many equilibria, which makes the analysis intricate. We first prove the non-existence of {\it ultimately monotone} traveling waves by applying the phase plane analysis. Next, we study the large time behavior of the solution of the Cauchy problem with a compactly supported initial datum. We not only reveal that the \lq\lq faster'' species excludes the \lq\lq slower'' 
one (with a known {\it spreading speed}), but also provide a sharp description of the profile of the solution, thus shedding light on  a new 
{\it{bump phenomenon}}.\\

\noindent{\underline{Key Words:} competition-diffusion system, traveling wave, Cauchy problem, large time behavior, bump phenomenon.}\\

\noindent{\underline{AMS Subject Classifications:}  35K57 (Reaction-diffusion equations), 35C07 (Traveling wave solutions), 35B40 (Asymptotic behavior of solutions).}
\end{abstract}

\section{Introduction}\label{s:intro}

We consider the Lotka-Volterra competition-diffusion system 
\begin{equation}\label{system}
\left\{
\begin{aligned}
&\partial_tu=u_{xx}+u(1-u-v), & t>0, x\in\mathbb{R},\\
&\partial_tv=dv_{xx}+rv(1-v-u), & t>0, x\in \R,
\end{aligned}
\right.
\end{equation}
which is {\it critical} among systems in the form of \eqref{system ab}. The main difficulty is that the underlying ODE  competition system
\begin{equation}\label{ode system}
\left\{
\begin{aligned}
&u'=u(1-u-v), & t>0,\\
&v'=rv(1-v- u), & t>0,
\end{aligned}
\right.
\end{equation}
admits infinitely many (nontrivial) equilibria: the whole line $u+v=1$. Because of that, there are very few available mathematical results on system \eqref{system}. In the present paper, we fill this gap by proving the non-existence of {\it ultimately monotone} traveling waves, and giving a very precise description of the large time 
behavior of the solution starting from a compactly supported initial datum, thus revealing a new {\it bump phenomenon}.

\medskip

In the absence of one species, system \eqref{system} reduces to the reaction-diffusion equation 
\begin{equation}\label{single kpp}
\partial _t u=du_{xx}+ru(1-u),\quad t>0, x\in\mathbb{R},
\end{equation}
introduced by Fisher \cite{Fisher} and Kolmogorov, Petrovsky and Piskunov 
\cite{KPP}  as a  population genetics model to investigate the propagation of a dominant gene in a homogeneous environment. The KPP equation \eqref{single kpp} has two main properties. Firstly, nonnegative traveling waves, corresponding to the ansatz $u(t,x)=U(x-ct)$ and  solving 
\begin{equation}\label{single tw}
\left\{
\begin{aligned}
&dU''+cU'+rU(1-U)=0\quad \text{ in } \mathbb{R},\\
&U(-\infty)=1,\ U(\infty)=0,
\end{aligned}
\right.
\end{equation}
exist if and only if their speeds $c\geq c^{*}:=2\sqrt{dr}$. Secondly, the solution of \eqref{single kpp} starting from a nonnegative (nontrivial) compactly supported initial datum, satisfies
\begin{equation*}
\begin{aligned}
&\lim_{t\to\infty}\sup_{|x|\ge ct}u(t,x)=0, & \text{ for all } c>c^{*},\\
&\lim_{t\to\infty}\sup_{|x|\le ct}|1-u(t,x)|=0, & \text{ for all } c<c^{*},
\end{aligned}
\end{equation*}
see \cite{Aronson Weinberger}. In other words, the minimal speed $c^*$ of traveling wave solutions corresponds to
the {\it spreading speed} of the solution of  the Cauchy problem with a compactly supported initial datum.

The general Lotka-Volterra competition-diffusion system is written 
\begin{equation}\label{system ab}
\left\{
\begin{aligned}
&\partial_tu=u_{xx}+u(1-u-a v), & t>0, x\in\mathbb{R},\\
&\partial_tv=dv_{xx}+rv(1-v-b u), & t>0, x\in \R.
\end{aligned}
\right.
\end{equation}
Here $u=u(t,x)$ and $v=v(t,x)$ represent the population densities of two 
competing species, $d>0$ and $r>0$ stand for the diffusion rate and intrinsic growth rate of $v$ (while those of $u$ have been normalized), and $a>0$ and $b>0$ represent the strength of $v$ and $u$, respectively, as competitors. The parameters $a$ and $b$ 
determine the behavior of the underlying ODE system (see below) but, once 
fixed, the outcomes for system \eqref{system ab} are highly dependent on the parameters $r$, $d$ and the initial datum.  The situation is therefore very rich and we refer to the works mentioned below for more details and references.

The so-called {\it weak competition case} corresponds to  $a<1$ and $b<1$. Nontrivial solutions of the underlying ODE system tend to the co-existence equilibrium.  For the diffusion system, it was proved by Tang and Fife \cite{Tang} that there exists a  minimal speed $c^\star>0$ such that a monotone traveling wave solution connecting the co-existence equilibrium to the null state $(0,0)$ exists if and only if $c\geq c^\star$, which is comparable to the Fisher-KPP equation mentioned above. Concerning the 
large time behavior of the Cauchy problem, some first estimates were obtained by Lin and Li \cite{Lin Li}. More recently, Liu, Liu and Lam 
\cite{Liu Liu Lam 1, Liu Liu Lam 2} obtained some rather complete results.

The so-called {\it strong competition case} corresponds to  $a>1$ and $b>1$. Since the co-existence equilibrium is unstable and the equilibria $(1,0)$ and $(0,1)$ are both stable for the underlying ODE system, this case 
corresponds to a {\it bistable} situation. For the diffusion system, it was proved by Kan-On \cite{Kan-On}, see also \cite{Gardner}, that there exists a unique traveling wave solution connecting $(1,0)$ to $(0,1)$. The sign of the speed of this wave determines the \lq\lq winner'' between $u$ 
and $v$, and thus is very relevant for  applications, see the review \cite{Girardin}. We refer to \cite{Girardin Nadin}, \cite{Guo Lin} and \cite{Rodrigo} for some results on this delicate issue. As far as the large time behavior of the Cauchy problem is concerned, we refer to the recent work of Carrere \cite{Carrere} revealing the possibility of {\it propagating terraces}, see \cite{pp2,pp}. Very recently, Peng, Wu and Zhou \cite{Peng Wu Zhou} provided  refined estimates of both the spreading speed and the profile of the solution. 

Last, the so-called {\it strong-weak competition case} corresponds to $a<1<b$. Nontrivial solutions of the underlying ODE system tend to the state 
$(1,0)$ meaning that \lq\lq $u$ excludes $v$''. For the diffusion system, 
the traveling wave solutions were constructed by Kan-On \cite{Kan-On monotone}. Concerning the large time behavior of the solution of the Cauchy problem,  
Girardin and Lam \cite{Girardin Lam} recently studied the spreading speed of solutions with an initial datum that is null (or exponentially decaying) on the right half line. They obtained a rather complete understanding of the spreading properties, revealing in particular the possibility of an {\it acceleration phenomenon} (see Appendix of the present paper for more details). 

In the present paper, our goal is to complete the above picture by considering the issues of both traveling wave solutions and the Cauchy problem in the so-called {\it critical competition case} $a=b=1$,  corresponding to system \eqref{system}.

\section{Main results}\label{s:results}

A traveling wave solution of system \eqref{system} is defined as follows.

\begin{definition}[($\alpha,\beta$)-traveling wave]\label{def:tw} Let $0\leq \alpha, \beta \leq 1$ be given  with $\alpha\neq \beta$. Then an $(\alpha,\beta)$-traveling wave solution (or traveling wave if there is no ambiguity) of \eqref{system} is a triplet $(c,U,V)$, where $c\in \R$ is the traveling wave speed and $(U,V)$ two nonnegative  profiles, solving
\begin{equation}\label{tw}
\left\{
\begin{aligned}
&U''+cU'+U(1-U-V)=0,\\
&dV''+cV'+rV(1-V-U)=0,\\
&(U,V)(-\infty)=(\alpha,1-\alpha),\\
&(U,V)(+\infty)=(\beta,1-\beta).
\end{aligned}
\right.
\end{equation}
\end{definition}

As mentioned above, for both the strong competition case and the strong-weak competition case, {\it monotone} traveling waves connecting $(1,0)$ to $(0,1)$ are known to exist. In the critical competition case under consideration, our first main result is that there is no {\it ultimately monotone} traveling wave connecting any two different nonnegative steady states on the line $u+v=1$.

\begin{definition}[Ultimately monotone ($\alpha,\beta$)-traveling wave]\label{def:um-tw} Let $0\leq \alpha, \beta \leq 1$ be given  with $\alpha\neq \beta$. Then an ultimately monotone $(\alpha,\beta)$-traveling wave solution is an $(\alpha,\beta)$-traveling wave for which there further exist $-\infty<z_0\le z^*_0<+\infty$ such that 
\begin{equation}\label{monotone-assumption}
U'(z)V'(z)\neq 0,\quad\text{for all } z\in(-\infty,z_0]\cup[z^*_0,+\infty).
\end{equation}
\end{definition}

\begin{remark}\label{rem:minus-speed} Obviously, if $(c,U(z),V(z))$ is an (ultimately monotone) $(\alpha,\beta)$-traveling wave then $(-c,U(-z),V(-z))$ is a (ultimately monotone) $(\beta,\alpha)$-traveling wave.
\end{remark} 

In other words, we do not require the traveling wave to be monotone on $\R$, but only to be monotone in some neighborhoods of both $-\infty$ and $+\infty$. This reinforces our non-existence result which states as follows.

\begin{theorem}[Non-existence of ultimately monotone traveling waves]\label{th:no tw alpha<1}  
 Let $0\leq \alpha, \beta \leq 1$ be given  with $\alpha\neq \beta$. Then, 
 there is no ultimately monotone $(\alpha,\beta)$-traveling wave for system \eqref{system}.
\end{theorem}

The above theorem is proved in Section \ref{s:no tw}. The starting point consists in transforming system \eqref{tw} into 
a first order system of four ODEs. Then, by a phase plane analysis, we prove that $U+V-1$ has to \lq\lq oscillate''  in a neighborhood of  $-\infty$ or $+\infty$, from which we get a contradiction.

\begin{remark}\label{rm: on tw TH}
As easily seen from the proof, to exclude the existence of a traveling wave with speed $c>0$, $c<0$, it is enough to assume that \eqref{monotone-assumption} holds in a neighborhood of $-\infty$, $+\infty$ respectively. In other words, there is no traveling wave for which the {\it invading state} is monotonically reached. The existence of a traveling wave for which the invading state is not monotonically reached remains an open issue. Last, as seen from subsection \ref{ss:apriori}, the non existence of standing waves ($c=0$) does not require any ultimately monotonicity assumption. 
\end{remark}

\medskip

Our second main focus is concerned with the large time behavior of the solution of system \eqref{system} starting from a nonnegative (nontrivial) compactly supported initial datum. In both the strong competition case \cite{Carrere}, \cite{Peng Wu Zhou}, and the strong-weak competition case \cite{Girardin Lam}, the monotone traveling wave solutions of the entire system play 
a key role in studying the large time behavior of the solution of the Cauchy problem. However, for the critical competition case, such traveling wave solutions do not exist.

In order to state our result, we define the (minimal) Fisher-KPP  traveling wave solution $(c_u,U_{KPP})$ as
\begin{equation}\label{def: single u tw}
c_u:=2, \quad \left\{
\begin{aligned}
&U_{KPP}''+c_uU_{KPP}'+U_{KPP}(1-U_{KPP})=0,\\
&U_{KPP}(-\infty)=1,\ U_{KPP}(\infty)=0,
\end{aligned}
\right.
\end{equation}
and, similarly, $(c_v,V_{KPP})$ as
\begin{equation}\label{def: single v tw}
c_v:=2\sqrt{dr}, \quad \left\{
\begin{aligned}
&dV_{KPP}''+c_vV_{KPP}'+rV_{KPP}(1-V_{KPP})=0,\\
&V_{KPP}(-\infty)=1,\ V_{KPP}(\infty)=0.
\end{aligned}
\right.
\end{equation}
Let us recall that both $U_{KPP}$ and $V_{KPP}$ are uniquely defined \lq\lq up to a shift''. Note that, $c_u$ (resp. $c_v$) also represents the spreading speed of $u$ (resp. $v$) in the absence of $v$ (resp. $u$).

\begin{theorem}[Propagation phenomenon]\label{th: profile}
Let $(u,v)=(u,v)(t,x)$ be the solution of system \eqref{system} starting from an initial datum $(u_0,v_0)=(u_0,v_0)(x)$ satisfying 
\begin{equation}\label{initial data}
\text{ $u_0$ and $v_0$ are continuous, nontrivial, compactly supported, and $0\leq u_0, v_0\leq 1$}.
\end{equation}
Then the following holds.
\begin{itemize}
    \item[$(i)$] Assume $dr>1$ (i.e. $c_v>c_u$). Then
\begin{equation}\label{profile of the solution dr>1 large x}
\lim_{t\to\infty}\left(\sup_{x\in \R} \left |v(t,x)-V_{KPP}\left(\vert x\vert-c_vt+\frac{3d}{c_v}\ln t+\eta_*(t)\right)\right|+\sup_{x\in \R}\ u(t,x)\right)=0,
\end{equation}
where $\eta_*$ is a bounded function on $[0,\infty)$.  
    \item[$(ii)$]  Assume $dr<1$ (i.e. $c_v<c_u$). Then
\begin{equation}\label{profile of the solution dr<1}
\lim_{t\to\infty}\left(\sup_{x\in \R} \left |u(t,x)-U_{KPP}\left(\vert x\vert-c_ut+\frac{3}{c_u}\ln t+\eta_{**}(t)\right)\right |+\sup_{x\in \R}\ v(t,x)\right)=0,
\end{equation}
where $\eta_{**}$ is a bounded function on $[0,\infty)$. 
\end{itemize}
\end{theorem}

The above theorem is proved in Section \ref{s:cauchy}.
Let us briefly comment on Theorem \ref{th: profile}. First of all, for the case $dr>1$ (or $dr<1$), the \lq\lq faster species'', namely $v$, excludes the \lq\lq slower one'', namely $u$, and imposes its spreading 
speed, see \eqref{profile of the solution dr>1 large x}. Furthermore, we find that the profile of the solution uniformly converges to the corresponding minimal KPP traveling wave solution, and this up to an identified logarithmic Bramson correction,  see \eqref{profile of the solution dr>1 large x} again. 

On the other hand, for the case $dr=1$, it is difficult to decide which species leads the invading front, and the behavior of the solution is highly depending on both the parameters and the shape of the initial datum. For the case $d=r=1$, for any initial datum satisfying \eqref{initial data}, a coexistence phenomenon  happens. However, for the case $dr=1$ but $d\neq 1$, the behavior of the solution is much more intricate. Indeed, in this case,  the logarithmic  phase drifts for $u$ and $v$ are different 
and there is a  narrow region of width $O(\ln t)$ where the behaviors of $u$ and $v$ are difficult to \lq\lq anticipate''. This may cause  some subtle phenomena (both species driving the front or one excluding the other) and makes the mathematical analysis quite involved. We hope to address these issues in a future work.

Our second result on the Cauchy problem deals with the region \lq\lq $\vert x \vert \leq \ep _* t$'', where the profile of the solution is more of the \lq\lq Heat equation type''.

\begin{theorem}[Bump phenomenon]\label{th: bump}
Let $(u,v)=(u,v)(t,x)$ be the solution of system \eqref{system} starting from an initial datum $(u_0,v_0)=(u_0,v_0)(x)$ satisfying \eqref{initial data}. Denote
$$
k^*:=\min\left(\frac{1}{2d},\frac{d}{2}\right), \quad d^*:=\max(1,d).
$$
Then the following holds.
\begin{itemize}
	\item[$(i)$] Assume $dr>1$ (i.e. $c_v>c_u$). Then
for $\ep_*>0$ small enough and $0<\theta<\frac{1}{2}$,  there exist $C_2>C_1>0$ and $T>0$ such that both
\begin{equation}\label{profile  u the solution dr>1 small x}
C_1t^{-\frac{1}{2}}e^{-\frac{x^2}{4t}}\le u(t,x)\le C_2t^{-k^*}e^{-\frac{x^2}{4d^*t}},\ 
\end{equation}
\begin{equation}\label{profile of v the solution dr>1 small x}
\max\left(C_1t^{-\frac{1}{2}}e^{-\frac{x^2}{4t}}-t^{-(1+\theta)},0\right)\le 1- v(t,x)\le C_2t^{-k^*}e^{-\frac{x^2}{4d^*t}},
\end{equation}
hold for any $t\geq T$, $\vert x\vert \leq \ep_*t$. 
    \item[$(ii)$]  Assume $dr<1$ (i.e. $c_v<c_u$). Then
for $\ep_{**}>0$ small enough and $0<\theta<\frac{1}{2}$, there exist $C_4>C_3>0$ and  $T>0$ such that both
\begin{equation}\label{profile  u the solution dr<1 small x}
C_3t^{-\frac{1}{2}}e^{-\frac{x^2}{4t}}\le v(t,x)\le C_4t^{-k^*}e^{-\frac{x^2}{4d^*t}},
\end{equation}
\begin{equation}\label{profile of v the solution dr<1 small x}
\max\left(C_3t^{-\frac{1}{2}}e^{-\frac{x^2}{4t}}-t^{-(1+\theta)},0\right)\le 1-u(t,x)\le C_4t^{-k^*}e^{-\frac{x^2}{4d^*t}},
\end{equation}
hold for any $t\geq T$, $\vert x\vert \leq \ep_{**}t$.
\end{itemize}
\end{theorem}

The  above theorem is proved in  Section \ref{s:bump}. Let us briefly comment on Theorem \ref{th: bump}, say in the case  $dr>1$. As revealed by \eqref{profile  u the solution dr>1 small x} and \eqref{profile of v the solution dr>1 small x}, the solution converges  to $(0,1)$ {\it exponentially} in regions of the form $\vert x\vert\ge \ep t$ with $\ep>0$, but only  {\it algebraically} in \lq\lq sublinear regions'' of the form $\vert x\vert \lesssim \sqrt{t}$. We call this a {\it{bump  phenomenon}}, see Figure \ref{figure:bump}. 
Such a phenomenon does not occur in the strong competition case \cite{Peng Wu Zhou}. As far as  the strong-weak competition case is concerned, the 
results as stated in  \cite{Girardin Lam} are not sufficient to decide if it occurs or not, but we assert it does not, as proved in the forthcoming work \cite{Xiao-Zhou}.  Therefore, the present paper is the first one revealing a bump phenomenon in the context of competition-diffusion systems. We believe such a phenomenon is reserved  for the critical case $a=b=1$, and is rare to happen in the context of homogeneous reaction-diffusion equations.

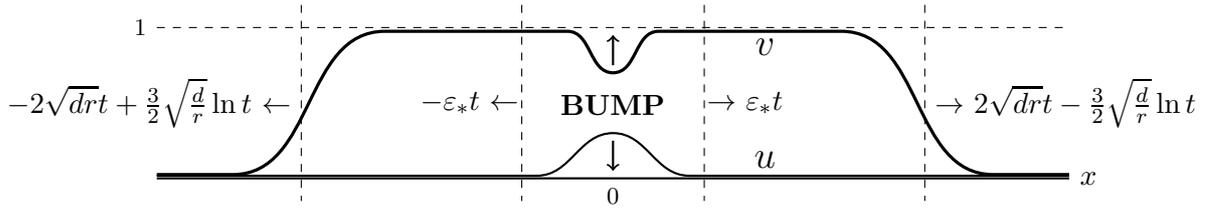
\begin{figure}
\begin{center}
\begin{tikzpicture}[scale = 1]
\draw[thick](-6,0) -- (6,0) node[right] {$x$};
\draw[->, thick] (0,1.5)--(0,1.9);
\draw[->, thick] (0,0.5)--(0,0.1);
\draw[dashed] (-6,2) -- (6,2);
\draw[thick]  (-6,0.035)--(-1,0.035)  to [out=0,in=180]  (0,0.6) to [out=0,in=180] (1,0.035)--(6,0.035);
\draw[very thick] (-6,0.05)--(-5,0.05) to [out=0, in=180]  (-3,1.95)--(-0.6,1.95) to [out=0, in=180] (0,1.4) to [out=0,in=180] (0.6,1.95)--(3,1.95) to [out=0,in=180] (5,0.05)--(6,0.05);
\draw[dashed] (4.1,2.3)--(4.1,-0.3);
\draw[dashed] (-4.1,2.3)--(-4.1,-0.3);
\node[left] at (-6,2) {\footnotesize{1}};
\node[below] at (0,0) {\footnotesize{0}};
\node[above] at (2,1.5) {\Large{$v$}};
\node[above] at (2,0) {\Large{$u$}};
\node[left] at (-4.1,1) {$-2\sqrt{dr}t+\frac{3}{2}\sqrt{\frac d r}\ln t$\ \textbf{$\leftarrow$}};
\node[right] at (4.1,1) {\textbf{$\rightarrow$}\ $2\sqrt{dr}t-\frac{3}{2}\sqrt{\frac  d r}\ln t$};
\draw[dashed] (1.2,2.3)--(1.2,-0.3);
\draw[dashed] (-1.2,2.3)--(-1.2,-0.3);
\node[above] at (0,0.7) {\bf{BUMP}};
\node[left] at (-1.1,1) {$-\ep_*t$\ \textbf{$\leftarrow$}};
\node[right] at (1.1,1) {\textbf{$\rightarrow$}\ $\ep_*t$};
\end{tikzpicture}
\caption{The asymptotic profile of the solution, in the
case  $dr>1$.}
\label{figure:bump}
\end{center}
\end{figure}

\section{Non-existence of traveling waves}\label{s:no tw}

This section is devoted to the proof of Theorem \ref{th:no tw alpha<1} on 
the non-existence of ultimately monotone traveling waves for system \eqref{system}. 

\subsection{Preliminary results and observations}\label{ss:apriori}

In this subsection, the ultimately monotonicity assumption is not required, and thus a traveling wave 
is understood in the sense of Definition \ref{def:tw}.  We start with the following {\it a priori} estimates for a traveling wave.

\begin{lemma}\label{lem:elliptic} Any traveling wave has to satisfy $0<U<1$, $0<V<1$, and $U'(\pm \infty)=V'(\pm \infty)=0$.
\end{lemma}

\begin{proof}
The positivity of the profiles follows from the strong maximum principle. 
 If $U\leq 1$ is not true, then $U$ has to reach a maximum value strictly 
larger than 1 at some point, and evaluating the  $U$-equation at this point gives a contradiction. Hence $U\leq 1$ and, from the strong maximum principle, $U<1$. Similarly, one has $V<1$.

We now prove the limit behavior $U'(+\infty)=0$, the other ones being proved similarly. Denote the set of
accumulation points of $U'$ in $+\infty$ by $\mathcal A$. Since $U$ is bounded,
$0\in \mathcal A$. Let $\ell \in \mathcal A$. Then there exists a sequence $z_n\to+\infty$ such that $U'(z_n)\to \ell$ as $n\to+\infty$. Then $(U_n,V_n)(z):=(U,V)(z+z_n)$ solves
$$
U_n''+cU_n'=-U_n(1-U_n-V_n).
$$
Since the $L^\infty$ norm of the right hand side term is
uniformly bounded with respect to $n$, the interior elliptic
estimates imply that, for all $R>0$ and $1<p<\infty$, the
sequence $(U_n)$ is bounded in $W^{2,p}(-R,R)$. From Sobolev
embedding theorem we have that, up to a subsequence, $U_{n}$ converges to some
$U_\infty$  in $C^{1}_{loc}(\R)$. The boundary condition $U(+\infty)=\beta$ thus enforces $U_\infty\equiv \beta$ and $U_\infty '\equiv 0$. As a result, $
U'(z_n)=U_n'(0)\to U'_\infty(0)=0$,  and thus $\ell =0$. Hence $\mathcal A=\{0\}$, which concludes the proof.
\end{proof}

We conclude this subsection by showing the non-existence of traveling wave solutions for two special cases.

\begin{proposition}[Non-existence of standing waves]\label{prop:no standing}
There is no standing wave, i.e. traveling wave with speed $c=0$, for system \eqref{system}.
\end{proposition}

\begin{proof} Assume $c=0$. By adding the both sides of the $U$-equation and the $V$-equation, we find that $W:=U+V$ satisfies
$$
W''+\left(U+\frac r d V\right)(1-W)=0, \quad W(\pm \infty)=1.
$$
If $W\not \equiv 1$, then $W$ reaches either a maximum value strictly larger than 1 or a minimum value in $(0,1)$, which is impossible from the above equation (recall that $U+\frac r d V>0$). As a result $W=U+V\equiv 1$. Going back to the original equations we have $U''=V''=0$.
Since $U$ and $V$ are bounded, $U$ and $V$ must be constant, which cannot 
happen since $\alpha \neq \beta$.
\end{proof}

\begin{proposition}[Non-existence of traveling waves when $d=1$]\label{prop:d=1} Assume $d=1$. Then there is no traveling wave for system \eqref{system}.
\end{proposition}

\begin{proof} Again, by adding the both sides of the $U$-equation and $V$-equation, we see that $W:=U+V$ satisfies
$$
W''+cW'+(U+rV)(1-W)=0, \quad W(\pm \infty)=1,
$$
so that, as in Proposition \ref{prop:no standing}, we have $W\equiv 1$, and thus $U''+cU'=0$. From Lemma \ref{lem:elliptic}, by integrating both sides from $-\infty$ to $+\infty$, we have $c(\beta-\alpha)=0$,
which yields a contradiction since $c\neq 0$ and $\alpha\neq \beta$. 
\end{proof}

\subsection{Proof of Theorem \ref{th:no tw alpha<1}}\label{ss:no-tw-alpha-beta}

We now consider the case of ultimately monotone traveling waves.

\begin{lemma}\label{lm:1}
Let $0\leq \alpha, \beta \leq 1$ be given  with $\alpha\neq \beta$. Let $(c,U,V)$ be an ultimately monotone $(\alpha,\beta)$-traveling wave solution. Then \eqref{monotone-assumption} is refined in 
\begin{equation}\label{monotone-assumption-bis}
U'(z)V'(z)< 0,\quad\text{for all } z\in(-\infty,z_0]\cup[z^*_0,+\infty).
\end{equation}
\end{lemma}

\begin{proof} We only deal with the behavior around $-\infty$.  If the conclusion is false, we may assume that $U'(z)>0$ and $V'(z)>0$ for all $z\in(-\infty,z_0]$, the case $U'(z)<0$ and $V'(z)<0$ being treated similarly. From the boundary conditions $(U,V)(-\infty)=(\alpha,1-\alpha)$ and $(U,V)(+\infty)=(\beta,1-\beta)$, there must exist a  point $z_1>z_0$ such that
$$
U'(z)>0, V'(z)>0, \quad \text{for all } z<z_1, \quad U'(z_1)=0 \text{ or } V'(z_1)=0.
$$
Assume w.l.o.g. that $U'(z_1)=0$. In particular $U''(z_1)\leq 0$. From the $U$-equation, this enforces $(U+V)(z_1)\leq 1$ which contradicts to $(U+V)(-\infty)=1$ and $(U+V)'>0$ on $(-\infty,z_1)$.
\end{proof}

We now prove, in the case $d\neq 1$, the non-existence of ultimately monotone traveling waves with speed $c\neq 0$. In view of subsection \ref{ss:apriori}, this is enough to  complete the proof of Theorem \ref{th:no tw alpha<1}.

\begin{proof}[Completion of the proof of Theorem \ref{th:no tw alpha<1}] For $d\neq 1$, let us consider $(c,U,V)$  an ultimately monotone $(\alpha,\beta)$-traveling wave with $c\neq 0$. In the sequel we only deal  with the case $c>0$ for which we perform a  phase plane analysis around $-\infty$ (for the case $c<0$, one has to perform a phase plane analysis around $+\infty$ with similar arguments). We define $W:=\alpha-U$, $P:=U'$, $R:=V-1+\alpha$, $Q:=V'$. 
Then we can rewrite (\ref{tw}) as 
\begin{equation}\label{ode system alpha<1}
\left\{
\begin{aligned}
&W'=-P,\\
&P'=-cP-(\alpha-W)(W-R),\\
&R'=Q,\\
&Q'=-\frac{c}{d}Q-\frac{r}{d}(R+1-\alpha)(W-R),\\
&(W,P,R,Q)(-\infty)=(0,0,0,0),\\
&(W,P,R,Q)(+\infty)=(\alpha-\beta,0,\alpha-\beta,0).
\end{aligned}
\right.
\end{equation}

Assume that $W-R$ is ultimately nonnegative, that is
\begin{equation}
\label{not-u-pos}
\exists z^*<z_0, \forall z\leq z^*, (W-R)(z)\geq 0.
\end{equation}
From \eqref{monotone-assumption-bis}, we know that it holds either $P>0$ or $Q>0$ on $(-\infty, z_0]$. Moreover, from Lemma \ref{lem:elliptic}, $\alpha-W\ge 0$ and $R+1-\alpha\ge 0$. If $P>0$ on $(-\infty,z_0]$, then from the $P$-equation in \eqref{ode system alpha<1}, we have $P'< 0$ on $(-\infty, z^*]$, which contradicts to $P>0$ on $(-\infty,z_0]$ and $P(-\infty)=0$. If $Q>0$ on $(-\infty, z_0]$, we similarly get a contradiction from the $Q$-equation. Hence \eqref{not-u-pos} does not hold.

Assume that $W-R$ is ultimately nonpositive, that is
\begin{equation}
\label{not-u-neg}
\exists z^*<z_0, \forall z\leq z^*, (W-R)(z)\leq 0.
\end{equation}
From \eqref{monotone-assumption-bis} again, we know that it holds either $P<0$ or $Q<0$ on $(-\infty, z_0]$.  If $P<0$ on $(-\infty,z_0]$, then from the $P$-equation in \eqref{ode system alpha<1}, we have $P'> 0$ on $(-\infty, z^*]$, which contradicts to $P<0$ on $(-\infty,z_0]$ and $P(-\infty)=0$. If $Q<0$ on $(-\infty, z_0]$, we similarly get a contradiction from the $Q$-equation. Hence \eqref{not-u-neg} does not hold.

As a result, since $W-R$ is neither ultimately nonnegative nor ultimately nonpositive, we can  find a local maximum point $z_1<z_0$ and a local minimum point $z_2<z_0$ such that
$$(W-R)(z_1)>0,\ (W-R)'(z_1)=0, \ (W-R)''(z_1)\le 0;$$
$$(W-R)(z_2)<0,\ (W-R)'(z_2)=0,  \ (W-R)''(z_2)\ge 0.$$
Note that
$$ (W-R)''=cP+\frac{c}{d}Q+\left(\alpha-W+\frac{r}{d}(R+1-\alpha)\right)(W-R).$$
From \eqref{monotone-assumption-bis}, it holds either $Q>0$ or $Q<0$ on $(-\infty, z_0]$.
Let us first consider the case $d<1$. If $Q>0$ on $(-\infty,z_0]$, since $(W-R)'(z_1)=0$ means $(P+Q)(z_1)=0$, we have $(cP+\frac{c}{d}Q)(z_1)>0$. Therefore, from the above equation, $(W-R)''(z_1)>0$, which is a contradiction. On the other hand,  if $Q<0$ on $(-\infty,z_0]$, since $(P+Q)(z_2)=0$, we have $(cP+\frac{c}{d}Q)(z_2)<0$. Therefore, from the above equation, $(W-R)''(z_2)<0$, which is a contradiction. Last, the case $d>1$ can be treated similarly.

Therefore, we conclude that system \eqref{system} does not admit any ultimately monotone traveling wave.
\end{proof}

\section{The Cauchy problem}\label{s:cauchy}

In this section, we consider system (\ref{system}) with a compactly supported initial datum, and prove the propagation result, namely Theorem \ref{th: profile}. 

\subsection{Preliminaries}

Let us start by briefly recalling the competitive comparison principle. Define the operators
\begin{equation*}
N_1[u,v]:=u_t-u_{xx}-u(1-u-v)\quad \text{ and } \quad  N_2[u,v]:=v_t-dv_{xx}-rv(1-v-u).
\end{equation*}
Consider a domain $
\Omega:=(t_1,t_2)\times(x_1,x_2)$ with  $0\le t_1<t_2\le + \infty$ and $-\infty\le x_1<x_2\le+\infty$. A (classical) super-solution is a pair  $(\overline{u},\underline{v})\in \Big[C^1\Big((t_1,t_2),C^2((x_1,x_2))\Big)\cap C_b\left(\overline \Omega\right)\Big]^2$
satisfying
$$
N_1[\overline{u},\underline{v}]\geq 0 \quad  \text{and}\quad  N_2[\overline{u},\underline{v}]\leq 0\;  \text{ in } \Omega.
$$
Similarly, a (classical) sub-solution $(\underline u, \overline v)$ requires  $N_1[\underline{u},\overline{v}]\leq 0$ and $N_2[\underline{u},\overline{v}]\geq 0$. 

\begin{proposition}[Comparison Principle]\label{prop: cp}
Let $(\overline{u},\underline{v})$ and $(\underline{u},\overline{v})$ be a super-solution and sub-solution of system (\ref{system}) in $\Omega$, respectively. If
$$
\overline{u}(t_1,x)\ge \underline{u}(t_1,x) \quad \text{and} \quad \underline{v}(t_1,x)\le\overline{v}(t_1,x),\quad\text{for all } x\in (x_1,x_2),
$$
and, for $i=1,2$, 
$$
\overline{u}(t,x_i)\ge \underline{u}(t,x_i) \quad \text{and} \quad \underline{v}(t,x_i)\le\overline{v}(t,x_i),\quad\text{for all } t\in(t_1,t_2),
$$
then, it holds
$$ 
\overline{u}(t,x) \ge\underline{u}(t,x) \quad \text{ and } \quad \underline{v}(t,x)\le\overline{v}(t,x),\quad  \text{for all } (t,x)\in\Omega.
$$
 If $x_1=-\infty$ or $x_2=+\infty$, the hypothesis on the corresponding boundary condition can be omitted.
\end{proposition}

Denote $(u,v)=(u,v)(t,x)$ as the solution of \eqref{system} starting from $(u_0,v_0)=(u_0,v_0)(x)$ satisfying \eqref{initial data}. Obviously, $(1,0)$ is a super-solution while $(0,1)$ is a sub-solution. It thus follows from \eqref{initial data}, the comparison principle and the strong maximum principle that
\begin{equation}\label{sandwich-0-1}
0<u(t,x)<1 \text{ and } 0<v(t,x)<1, \quad \text{for all }\ t>0, x\in \R.
\end{equation}

Actually, the comparison principle also holds for the so-called {\it generalized} sub- and super-solutions. This is a rather well-known fact, and we refer to the clear exposition  in \cite[subsection 2.1]{Girardin Lam} for more  details. In particular,  if $(\underline {u}_1,\overline{v})$ 	and $(\underline {u}_2, \overline{v})$ are both classical sub-solutions, then $(\max(\underline {u}_1,\underline{u}_2),\overline v)$ is a generalized sub-solution. Also, 
 if $(\underline u,\overline{v}_1)$ 	and $(\underline u, \overline{v}_2)$ are both classical sub-solutions, then $(\underline u,\min(\overline{v}_1,\overline{v}_2))$ is a generalized sub-solution.

\medskip

We now start the proof of Theorem \ref{th: profile}.  Observe that, by changing the variables as $x=\sqrt{\frac{d}{r}}y$ and $t=\frac{1}{r}s$, system \eqref{system} can be rewritten as 
\begin{equation*}
\left\{
\begin{aligned}
&\partial_sv=v_{yy}+v(1-v-u), & s>0, y\in\mathbb{R}, \\
&\partial_su=d^{-1}u_{xx}+r^{-1}u(1-u-v), & s>0, y\in\mathbb{R}.\\
\end{aligned}
\right.
\end{equation*}
Therefore, without loss of generality, we assume from now that $dr>1$, that is $c_v>c_u$, and shall prove the statement $(i)$ in Theorem \ref{th: profile}. 

\medskip

Since $c_v>c_u$ and $u$ cannot propagate faster than $c_u$, the behavior of the solution in the region $\vert x\vert >c_u t$ is rather well-understood.

\begin{proposition}[Estimates in the region $\vert x \vert >c_ut$]\label{prop:estimate on leading edge}
We have
\begin{equation}\label{truc1}
\lim_{t\to\infty}\sup_{|x|\ge c t} u(t,x) =0, \quad \text{ for all }  c>c_u,
\end{equation}
\begin{equation}\label{truc2}
\lim_{t\to\infty}\sup_{|x|\ge c t}v(t,x)=0, \quad \text{ for all }  c>c_v,
\end{equation}
and
\begin{equation}\label{profile on intermediate zone}
\lim_{t\to\infty}\sup_{c_1t\leq |x|\leq  c_2t} |1-v(t,x)| =0, \quad \text{ for all }  c_u<c_1<c_2<c_v.
\end{equation}
\end{proposition}

\begin{proof} Without loss of generality, we only deal with the case $x\ge 0$. Define 
$$
U(t,x):=C_1e^{-\frac{c_u}{2}(x-c_ut)}\quad \text{and} \quad V(t,x):=C_2e^{-\frac{c_v}{2d}(x-c_vt)},
$$
where $C_1>0$ and $C_2>0$ are chosen large enough so that $U(0,\cdot)\geq u_0$ and $V(0,\cdot)\geq v_0$. We can easily check that $(U,0)$ is a super-solution while $(0,V)$ is a sub-solution. As a result, we have 
\begin{equation}\label{easy}
0<u(t,x)\leq \min\left(1,C_1e^{-\frac{c_u}{2}(x-c_ut)}\right)\quad\text{and}\quad 0<v(t,x)\leq \min\left(1,C_2e^{-\frac{c_v}{2d}(x-c_vt)}\right),
\end{equation}
from which \eqref{truc1} and \eqref{truc2} follow. 

Next, let $c_u<c_1<c_2<c_v$ be given. Select $0<a<1<b$ and consider $(u^{*},v^{*})$  the solution of
the strong-weak  competition system
\begin{equation}\label{monostable competition system}
\left\{
\begin{aligned}
&\partial_t u^*=u^*_{xx}+u^*(1-u^*-av^*),\\
&\partial_t v^*=dv^*_{xx}+rv^*(1-v^*-bu^*),\\
\end{aligned}
\right.
\end{equation}
starting from $(u_0,v_0)$. Obviously, $(u^{*},v^{*})$ is a super-solution for system \eqref{system}, and thus $v^{*}(t,x)\leq v(t,x)\leq 1$ for all $t\ge 0$ and $x\in\R$. Since the statements $(2)$ and $(3)$ in \cite[Theorem 1.1]{Girardin Lam} imply
$$
\lim_{t\to\infty}\sup_{c_1t\leq x\leq c_2t}|1-v^*(t,x)|=0,
$$
 the same conclusion holds for $v$.
\end{proof}

\subsection{Construction of the super-solution}\label{ss:super}

The goal of this subsection is to construct an adequate super-solution in, roughly speaking, the region $\vert x\vert <c_u t$. More precisely, let $\frac 1 d<r_1<r$ be given and define $c^{*}_v:=2\sqrt{dr_1}<c_v$. In the sequel, we introduce $V_{1}$ as a traveling wave solution with speed $c_v^{*}=2\sqrt{d r_1 }$ solving
\begin{equation}\label{single-V1}
\left\{
\begin{aligned}
&dV_1''+c_v^{*}V_1'+r_1V_1(1-V_1)=0,\\
&V_1(-\infty)=1,\  V_1(\infty)=0.
\end{aligned}
\right.
\end{equation} 
As well-known, $V_1'<0$ and there are $\lambda_1>0$ and $M _1>0$ such that
\begin{equation}\label{estimate of V at infinity-un}
 1-V_1(\xi)\sim M _1 e^{\lambda _1 \xi}\; \text{ as } \xi \to -\infty.
\end{equation}
Let us fix  $c_u<c_1<c_v^{*}$. For $T>0$, we will work in the domain (which is \lq\lq expanding in time'')
\begin{equation}\label{def:omega-un}
\Omega _1(T):=\{(t,x)\in(T,\infty)\times \R: \vert x\vert <c_1t\}.
\end{equation}
It turns out that the construction of the super-solution is highly dependent on the value of $d$.

\medskip

\noindent{$\bullet$  \bf The case $d\le 1$.} We introduce $s=s(t,x)$ as the solution of the Cauchy problem
\begin{equation}
\label{def-s}
\left\{
\begin{aligned}
\partial_t s&= s_{xx},\\
s(0,x)&=s_0(x):=B_1e^{-q|x|},
\end{aligned}
\right.
\end{equation}
and look for a super-solution $(\tilde U,\tilde V)$ in the form
\begin{equation}\label{definition of super sol system}
\left\{
\begin{aligned}
\tilde{U}(t,x)&:=t^{\frac{1-d}{2}}(1-e^{-\tau t})s(t,x),\\
\tilde{V}(t,x)&:=V_{1}(x-c^{*}_vt)+V_{1}(-x-c^{*}_vt)-1-\tilde{U}(t,x).
\end{aligned}
\right.
\end{equation}
All parameters that will be determined below (namely $B_1$, $q$ and $\tau$) are positive, and $q<1$.

\medskip

\noindent{$\bullet$ \bf The case $d\ge 1$.} We  introduce $s=s(t,x)$ as the solution of the Cauchy problem
\begin{equation}
\label{def-s*}
\left\{
\begin{aligned}
\partial_t s&= ds_{xx},\\
s(0,x)&=s_0(x):=B_1e^{-q|x|},
\end{aligned}
\right.
\end{equation}
and look for a super-solution $(\tilde U,\tilde V)$ in the form
\begin{equation}\label{definition of super sol system d>1}
\left\{
\begin{aligned}
\tilde{U}(t,x)&:=t^{\frac{d-1}{2d}}(1-e^{-\tau t})s(t,x),\\
\tilde{V}(t,x)&:=V_{1}(x-c^{*}_vt)+V_{1}(-x-c^{*}_vt)-1-\tilde{U}(t,x).
\end{aligned}
\right.
\end{equation}
All parameters that will be determined below (namely $B_1$, $q$ and $\tau$) are positive, and $q<\frac{1}{d}$.

\medskip

Obviously, \eqref{def-s}---\eqref{definition of super sol system} and \eqref{def-s*}---\eqref{definition of super sol system d>1} coincide when $d=1$. 

\begin{proposition}[Super-solutions]\label{prop: inequality super solution} The following holds.
\begin{itemize}
\item[$(i)$] Assume $d\le 1$. Let $0<q<1$ and $0<\tau<\lambda_1(c^{*}_v-c_1)$ be given. Then there exists $T^{*}>0$ such that,  for all $B_1>0$, $(\tilde U,\tilde V)$, given by \eqref{def-s}---\eqref{definition of super sol system},  is  a super-solution in the domain $\Omega_1(T^{*})$ as defined in \eqref{def:omega-un}.

\item[$(ii)$] Assume $d\ge 1$.  Let $0<q<\frac{1}{d}$ and $0<\tau<\lambda_1(c^{*}_v-c_1)$ be given. Then there exists $T^{*}>0$ such that,  for all $B_1>0$, $(\tilde U,\tilde V)$, given by \eqref{def-s*}---\eqref{definition of super sol system d>1},  is  a super-solution in the domain $\Omega_1(T^{*})$ as defined in \eqref{def:omega-un}.
\end{itemize}
\end{proposition}

\begin{proof} Since our super-solutions  are even functions, it is enough to deal with $x\geq 0$. In other words, we work for $t\geq T$ (with $T>0$ to be selected) and $0\leq x< c_1t$, with $c_u<c_1<c_v^{*}$.  For ease of notations, we shall use the shortcuts  $\xi_{\pm}:=\pm x-c^{*}_vt$. Since $\xi_-\leq -c_v^{*}t$ and $\xi_+\leq -(c_v^{*}-c_1)t$, it follows  from $V_1'<0$ and \eqref{estimate of V at infinity-un}  that there exist $C_->0$ and $C_+>0$ such that, for $T>0$ large enough,
\begin{equation}\label{estimate of 1-V super sol}
1-V_1(\xi_-)\leq C_-e^{-\lambda_1c^{*}_vt}\quad \text{ and }\quad  1-V_1(\xi_+)\leq C_+e^{-\lambda_1(c^{*}_v-c_1)t},\quad\text{for all}\ (t,x)\in \Omega_1^+(T),
\end{equation}
where $\Omega_1^+(T):=\Omega_1(T)\cap (T,\infty)\times [0,\infty)$. Moreover, up to enlarging $T>0$ if  necessary, there exists $0<\rho<\frac 1 3$ such that 
\begin{equation}
\label{rho-bis}
0<1-V_1(\xi_\pm)\leq \rho, \quad  \text{for all}\ (t,x)\in \Omega_1^+(T).
\end{equation}

\medskip

We first assume $d\le 1$. Some straightforward computations combined with  \eqref{def-s} yield
\begin{equation*}\label{super sol N1}
N_1[\tilde{U},\tilde{V}]=t^{\frac{1-d}{2}}(1-e^{-\tau t})s\left(\frac{1-d}{2}t^{-1}+\frac{\tau e^{-\tau t}}{1-e^{-\tau t}}-2+V_1(\xi_+)+V_1(\xi_-)\right).
\end{equation*}
In view of (\ref{estimate of 1-V super sol}), by choosing $\tau<\lambda_1(c_v^{*}-c_1)$, we deduce that, for $T>0$ large enough, $N_1[\tilde U,\tilde V]\geq 0$ in $\Omega_1^{+}(T)$. On the other hand, some straightforward computations combined with  (\ref{def-s}) and (\ref{single-V1}) yield
$$
N_2[\tilde U,\tilde V]=J_1+J_2+J_3,
$$
where 
\begin{eqnarray*}
J_1&:=&t^{\frac{1-d}{2}}(1-e^{-\tau t})s\left(r\Big(2-V_1(\xi_+)-V_1(\xi_-)-\frac{\tau e^{-\gamma t}}{r(1-e^{-\tau t})}\Big)-\frac{1-d}{2}t^{-1}-(1-d)\frac{\partial_ts}{s}\right),\\
J_2&:=&(r_1-r)V_1(\xi_+)(1-V_1(\xi_+)),\\
J_3&:=&(1-V_1(\xi_-))\left((r_1-r)V_1(\xi_-)+r(2-2V_1(\xi_+))\right).
\end{eqnarray*}
Since $r_1<r$ and $0<V_1<1$, we have $J_2\leq 0$. Next, from (\ref{rho-bis}), we have
$$
J_3\leq (1-V_1(\xi_-))\left((r_1-r)(1-\rho)+r(2-2V_1(\xi_+))\right),
$$
which, in view of (\ref{estimate of 1-V super sol}), is nonpositive up to enlarging $T>0$ if  necessary. Last, from the \lq\lq Heat kernel expression'' of $s(t,x)$, namely
$$
s(t,x)=\left(G(t,\cdot)*s_0\right)(x),\quad\text{where}\ G(t,x):=\frac{1}{\sqrt{4\pi t}}e^{-\frac{x^2}{4t}},
$$
we can check that $\partial _t s(t,x)\geq -\frac{1}{2t}s(t,x)$. As a result, since $d\leq 1$, we have
\begin{equation*}
J_1\le t^{\frac{1-d}{2}}(1-e^{-\tau t})s r\left(2-V_1(\xi_+)-V_1(\xi_-)-\frac{\tau e^{-\tau t}}{r(1-e^{-\tau t})}\right).
\end{equation*}
In view of (\ref{estimate of 1-V super sol}) and $\tau<\lambda_1(c_v^{*}-c_1)$, we have  $J_1\le 0$ up to enlarging $T>0$  if  necessary. As a result, $N_2[\tilde U,\tilde V]\leq 0$ in $\Omega_1^{+}(T)$.

\medskip
Next, we assume $d\ge 1$.
Some straightforward computations combined with  \eqref{def-s*} yield
\begin{equation*}\label{super sol N1 d>1}
N_1[\tilde{U},\tilde{V}]=t^{\frac{d-1}{2d}}(1-e^{-\tau t})s\left(\frac{d-1}{2d}t^{-1}+\frac{d-1}{d}\frac{\partial_ts}{s}+\frac{\tau e^{-\tau t}}{1-e^{-\tau t}}-2+V_1(\xi_+)+V_1(\xi_-)\right).
\end{equation*}
As above, since $d\ge 1$, $\partial _t s(t,x)\geq -\frac{1}{2t}s(t,x)$ implies 
$$N_1[\tilde{U},\tilde{V}]\ge t^{\frac{d-1}{2d}}(1-e^{-\tau t})s\left(\frac{\tau e^{-\tau t}}{1-e^{-\tau t}}-2+V_1(\xi_+)+V_1(\xi_-)\right).$$
In view of (\ref{estimate of 1-V super sol}) and $\tau<\lambda_1(c_v^{*}-c_1)$,  we deduce that, for $T>0$ large enough, $N_1[\tilde U,\tilde V]\geq 0$ in $\Omega_1^{+}(T)$. On the other hand, some straightforward computations combined with (\ref{def-s*}) and (\ref{single-V1}) yield
$$
N_2[\tilde U,\tilde V]=J_1+J_2+J_3,
$$
where 
\begin{eqnarray*}
J_1&:=&t^{\frac{d-1}{2d}}(1-e^{-\tau t})s\left(r(2-V_1(\xi_+)-V_1(\xi_-))-\frac{\tau e^{-\tau t}}{1-e^{-\tau t}}-\frac{d-1}{2d}t^{-1}\right),\\
J_2&:=&(r_1-r)V_1(\xi_+)(1-V_1(\xi_+)),\\
J_3&:=&(1-V_1(\xi_-))\left((r_1-r)V_1(\xi_-)+r(2-2V_1(\xi_+))\right).
\end{eqnarray*}
By applying the same argument as that for $d\le 1$, we get $N_2[\tilde U,\tilde V]\leq 0$ in $\Omega_1^{+}(T)$.
\end{proof}
 
Note that, time $T^*$ in Proposition \ref{prop: inequality super solution} is independent on  $B_1>0$, which leaves \lq\lq some room'' to enlarge $B_1$ so that the \lq\lq initial order'' and the  \lq\lq order on the boundary of the domain'' are  suitable for the comparison principle to be applicable.

\begin{proposition}[First estimate on $(u,v)$]\label{prop: boundary condition super sol} There exist $0<q<\min(1,\frac{1}{d})$, $T^{**}>0$ and $B_1>0$  such that
$$
u(t,x)\leq \tilde U(t,x)\quad \text{ and }\quad  \tilde V(t,x)\le v(t,x), \quad \text{ for all }\; t\geq T^{**}, \vert x\vert \leq c_1t,
$$
where $(\tilde U,\tilde V)$ is given by \eqref{def-s}---\eqref{definition of super sol system} when $d\leq 1$, and by \eqref{def-s*}---\eqref{definition of super sol system d>1} when $d\geq 1$. 
\end{proposition}

\begin{proof} We aim at applying the comparison principle in $\Omega_1(T)$, as defined in (\ref{def:omega-un}), with a well-chosen $T>0$. Select $0<q<\min(1,\frac 1 d)$ small enough so that
\begin{equation}
\label{choice-q}
\max(qc_1-q^2, qc_1-dq^2)<c_1-c_u.
\end{equation} 
From Proposition \ref{prop: inequality super solution}, for any $T\geq T^{*}$, we are equipped with a super-solution $(\tilde U,\tilde V)$ for which $B_1>0$ is arbitrary. We only deal with the case $d\leq 1$, the case $d\geq 1$ being similar.

We first  focus on  $ x =c_1t$, $t\geq T^*$ (the case $x=-c_1t$, $t\geq T^*$ being similar). Let us prove that, up to enlarging $T^*>0$ if necessary, it holds
\begin{equation}\label{cl:u< U on ct}
u(t,c_1t)\le \tilde{U}(t,c_1t), \quad \text{for all}\ t\geq T^*.
\end{equation}
From the proof of Proposition \ref{prop:estimate on leading edge}, we know that
$u(t,c_1t)\leq C_1e^{-(c_1-c_u)t}$ (recall that $c_u=2$).
Recalling $s_0(x)=B_1e^{-q\vert x\vert}$, we have
$$
s(t,x)=\frac{B_1}{\sqrt{4\pi t}}\left(\int _{-\infty}^0 e^{-\frac{(x-y)^2}{4t}}e^{qy}dy+\int_0^{+\infty}e^{-\frac{(x-y)^2}{4t}}e^{-qy}dy\right),
$$
which can be recast, after some elementary computations, 
\begin{equation}\label{expression of s}
s(t,x)=\frac{B_1}{\sqrt \pi}\left(e^{q^2t-qx}\int_{\frac{2qt-x}{2\sqrt{t}}}^{+\infty}e^{-w^2}dw+e^{q^2t+qx}\int_{\frac{2qt+x}{2\sqrt{t}}}^{+\infty}e^{-w^2}dw\right).
\end{equation}
In particular, since  $2q<2<c_1$, we have, by enlarging $T^*>0$ if  necessary,
$$
\tilde U(t,c_1t)\geq  \frac 12 s(t,c_1t)\geq \frac{B_1}{4}e^{-(qc_1-q^{2})t}\geq \frac{B_1}{4}e^{-(c_1-c_u)  t}.
$$
The last inequality holds from the choice \eqref{choice-q}. Thus $B_1>4C_1$ is enough to get (\ref{cl:u< U on ct}). 

Let us recall that $v\geq v^*$ where $(u^*,v^*)$ is the solution of the strong-weak competition system (\ref{monostable competition system}) with the same initial datum $(u_0,v_0)$.  From Lemma \ref{lm: decay estimamte u and v}  $(ii)$ (see Appendix), up to enlarging $T^*$ if necessary, there exist $\mu >0$ and $K>0$ such that $v^*(t,c_1t)\geq 1-Ke^{-\mu t}$ for all $t\geq T^{*}$. On the other hand, the construction of $\tilde{V}$ implies that $\tilde{V}(t,c_1t)\le 1-\tilde U(t,c_1t)$. Therefore, by choosing $qc_1-q^2<\mu$, up to enlarging $T^*>0$ if  necessary, we have
\begin{equation}\label{cl:v> V on ct}
\tilde V(t,c_1t)\leq v(t,c_1t), \quad \text{for all}\ t\geq T^*.
\end{equation}

Now, $q>0$ and $T^*>0$ are fixed from the above discussion. We focus on the initial datum, namely $t=T^{*}$, $\vert x\vert \leq c_1 T^{*}$. As above, we deduce from (\ref{expression of s}) that
$$
 \inf_{|x|\le c_1T^*}\tilde{U}(T^*,x)\ge \frac 12  s(T^*,c_1T^*)\geq \frac{B_1}{4}e^{-(qc_1-q^{2})T^*}\geq 1\geq \sup_{t>0, x\in \R} u(t,x),
 $$ 
 provided that $B_1>0$ is large enough. On the other hand, 
$$
\sup_{|x|\le c_1T^*}\tilde{V}(T^*,x)\leq  1- \inf_{|x|\le c_1T^*}\tilde{U}(T^*,x)\leq 1-\frac{B_1}{4}e^{-(qc_1-q^{2})T^*}\leq 0 \leq \inf _{t>0, x\in \R} v(t,x),
$$
provided that $B_1>0$ is large enough. 

As a consequence, the comparison principle can be applied in $\Omega_1(T^{*})$, which concludes the proof of  Proposition \ref{prop: boundary condition super sol}.
\end{proof}

\subsection{Proof of Theorem \ref{th: profile}}

From the discussion above, we are now in the position to obtain the following spreading speed result. 

\begin{proposition}[Spreading speed]\label{prop: spreading speed} Let $(u,v)=(u,v)(t,x)$ be the solution of \eqref{system} starting from an initial datum $(u_0,v_0)=(u_0,v_0)(x)$ satisfying (\ref{initial data}).  Then the following holds.
\begin{itemize}
	\item[$(i)$] Assume $dr>1$ (i.e. $c_v>c_u$). Then, for any $0<c_1<c_v<c_2$, 
\begin{equation}\label{spreading speed dr>1}
\lim_{t\to\infty}\left(\sup_{x\in\R}u(t,x)+\sup_{\vert x\vert \leq c_1t}|1-v(t,x)|+\sup_{\vert x\vert \geq c_2 t}v(t,x)\right)=0.
\end{equation}

    \item[$(ii)$] Assume $dr<1$ (i.e. $c_v<c_u$). Then, for any $0<c_3<c_u<c_4$, 
\begin{equation}\label{spreading speed dr<1}
\lim_{t\to\infty}\left(\sup_{x\in\R }v(t,x)+\sup_{\vert x\vert \leq c_3 t}|1-u(t,x)|+\sup_{\vert x\vert \geq c_4t}u(t,x)\right)=0.
\end{equation}
\end{itemize}
\end{proposition}

\begin{proof} Let us prove $(i)$. The result on $u$ in (\ref{spreading speed dr>1}) is obtained by combining (\ref{truc1}) and Proposition \ref{prop: boundary condition super sol}. Next, for a given $0<c_1<c_v$, we select $c_1<c_v^{*}<c_v$. Then,  Proposition \ref{prop: boundary condition super sol} yields $\sup_{\vert x\vert \leq c_1t} \vert 1-v(t,x)\vert \to 0$ as $t\to \infty$. The last part of (\ref{spreading speed dr>1}) is nothing else than the estimate (\ref{truc2}).
\end{proof}

We are now in the position to complete  the proof of Theorem \ref{th: profile}.

\begin{proof}[Proof of Theorem \ref{th: profile} $(i)$] Since the proof for 
$x\le 0$ follows from the same argument, we only deal with $x\ge 0$.  Let us prove \eqref{profile of the solution dr>1 large x}. For a given  $m\in(0,1)$, we define $E_m(t)$ as the $m$-level set of $v(t,\cdot)$, namely
$$
E_m(t):=\{x>0: v(t,x)=m\}.
$$
We claim that there exist $M>0$ and $T>0$ such that
\begin{equation}\label{Em}
 c_vt-\frac{3d}{c_v}\ln t-M\leq \min E_m(t) \leq \max E_m(t)\le c_vt-\frac{3d}{c_v}\ln t+M,\quad \text{ for all } t\ge T.
\end{equation}
The upper bound in \eqref{Em} is obtained by using  the solution of  $
\partial_t\overline{v}=d\overline{v}_{xx}+r\overline{v}(1-\overline{v})$, starting from $\overline v(0,x)=v_0(x)$, as a super-solution. We refer to \cite[Lemma 4.1]{Peng Wu Zhou}, see also \cite{Bramson} and \cite{log delay 3}. As for the lower bound in \eqref{Em}, it follows from \cite[Lemma 4.5]{Peng Wu Zhou} which is based on an idea of \cite{log delay 3}. A sketch of the proof is as follows. Let us a fix a small $\ep>0$ (this is necessary because of the bump phenomenon).  By combining  \eqref{easy}, Proposition \ref{prop: boundary condition super sol}  and \eqref{expression of s}\footnote{from which one can straightforwardly deduce that $\sup_{x\geq \ep t} s(t,x)=s(t,\ep t) =O\left(\frac{e^{-\frac{\ep ^2}{4}t}}{\sqrt t}\right)$.}, we see that  there exist $C>0$, $\mu>0$ and $T>0$ such that
\begin{equation}\label{exponential decay of u}
\sup_{|x|\ge \varepsilon t}u(t,x)\le Ce^{-\mu t},\quad \text{ for all } t\ge T.
\end{equation}
As a result, by setting $C_0=rC$, we have
$$
\partial_t v\geq d v_{xx}+v(r-r v-C_0e^{-\mu t}), \quad\text{for all } t>0, x>\ep t.
$$
The key idea, borrowed from  \cite{log delay 3}, is then to linearize the above equation, and to consider
\begin{equation}\label{eq w: low estimate of Em}
\partial_tw=dw_{xx}+w(r-C_0e^{-\mu t}), \quad t>0, x>\Gamma(t):=c_vt-\frac{3d}{c_v}\ln(t+t_0),
\end{equation}
together with the Dirichlet boundary conditions  $w(t,\Gamma(t))=0$ and a compactly supported initial datum $w(0,\cdot)$. Then, one can exactly reproduce the technical arguments of \cite[Lemma 4.3 and 4.4]{Peng Wu Zhou}, mainly borrowed from \cite{log delay 3}, to obtain the lower bound in \eqref{Em}. Last, by applying  \eqref{Em}, we can reproduce the argument of \cite[Section 4]{log delay 3}, see also \cite[Proof of Theorem 2]{Peng Wu Zhou}, to conclude that there exists a bounded function $\eta_*:[0,\infty)\to \R$ such that
\begin{equation}\label{profile of v with log delay}
\lim_{t\to\infty}\,\sup_{x\geq 0} \, \left|v(t,x)-V_{1}\left(x-c_vt+\frac{3d}{c_v}\ln t+\eta_*(t)\right)\right|=0,
\end{equation}
which, combined with \eqref{spreading speed dr>1}, concludes the proof of \eqref{profile of the solution dr>1 large x}.
\end{proof}

\section{The bump phenomenon}\label{s:bump}

In this section, we will provide a lower estimate for the solution of  system (\ref{system}) starting from a compactly supported initial  datum, and prove Theorem \ref{th: bump} on the bump phenomenon.

\subsection{Construction of the sub-solution}\label{ss:sub}

The goal of this subsection is to construct an adequate sub-solution in, roughly speaking, the region $\vert x\vert <c_v t$. More precisely, let $r_2>r$ be given and define $c^{**}_v:=2\sqrt{dr_2}>c_v$. Let us fix  $c_v<c_2<c_v^{**}$. For $T>0$, we will work in the domain (which is \lq\lq expanding in time'')
\begin{equation}\label{def:omega-zero}
\Omega _2(T):=\{(t,x)\in(T,\infty)\times \R: \vert x\vert <c_2t\}.
\end{equation}

A key observation for the construction is the following: from Proposition \ref{prop:estimate on leading edge}, in the region  $c_ut<|x|<c_vt$, we have $u+v\approx 1$, and therefore
\begin{equation*}
\left\{
\begin{aligned}
&\partial_tu\approx u_{xx},\\
&\partial_tv\approx dv_{xx}.
\end{aligned}
\right.
\end{equation*}
We thus introduce $f=f(t,x)$ and $h=h(t,x)$ the solutions of the Cauchy problems
\begin{equation}
\label{def-f-h}
\left\{
\begin{aligned}
\partial_t f&= f_{xx},\\
f(0,x)&=f_0(x):=B_2\mathbf 1 _{(-1,1)}(x),
\end{aligned}
\right.
\qquad 
\left\{
\begin{aligned}
\partial_t h&= h_{xx},\\
h(0,x)&=h_0(x):=B_3e^{-k|x|},
\end{aligned}
\right.
\end{equation}
and look for a sub-solution $(U,V)$ in the form
\begin{equation}\label{definition of sub sol system}
\left\{
\begin{aligned}
U(t,x)&:=g(t)f(t,x)-h(t,x),\\
V(t,x)&:=V_{2}(x-c^{**}_vt-\zeta_0)+V_{2}(-x-c^{**}_vt-\zeta_0)-1-U(t,x)+\frac 1{t^{1+\theta}},
\end{aligned}
\right.
\end{equation}
where
$$
g(t):=\exp \frac{1}{\delta(1+t)^{\delta}}.
$$
Here, all parameters that will be determined below (namely $B_2$, $B_3$, $k$, $\zeta _0$, $\theta$, $\delta$) are positive, $B_2<1$, $B_3<1$, while $V_{2}$ is the traveling wave solution with speed $c_v^{**}=2\sqrt{d r_2 }$ satisfying
\begin{equation}\label{single-V0}
\left\{
\begin{aligned}
&dV_2''+c_v^{**}V_2'+r_2V_2(1-V_2)=0,\\
&V_2(-\infty)=1,\  V_2(\infty)=0.
\end{aligned}
\right.
\end{equation}
It is well-known that $V_2'<0$ and there exist $\lambda_2>0$ and $M _2>0$ such that
\begin{equation}\label{estimate of V at infinity}
 1-V_2(\xi)\sim M_2 e^{\lambda _2 \xi}\; \text{ as } \xi \to -\infty.
\end{equation}

Next, we shall provide some estimates which are based on the \lq\lq Heat kernel expressions'' of the solutions $f$ and $h$ of \eqref{def-f-h}. Note that, in Lemma \ref{lm: some estimates of f and h}, $0<B_2<1$ and $0<B_3<1$ can be relaxed to $B_2>0$ and $B_3>0$.

\begin{lemma}\label{lm: some estimates of f and h}
Let $\delta>0$ and $k>0$ be given,  and set $B_3=\gamma B_2$ with some $\gamma>0$. Then the following holds.
\begin{itemize} 
\item[$(i)$] For any given $0<j<k$,
\begin{equation*}
h(t,x)\leq \frac{B_3}{\sqrt{\pi}}\frac{k}{k^2-j^2}t^{-\frac{1}{2}}e^{-\frac{x^2}{4t}}, \quad \text{ for all } t>0, \vert x\vert \leq 2jt.
\end{equation*}
\item[$(ii)$] For any given $0<j<k$ and $T>0$, there exists $\gamma_1>0$ such that, if $0<\gamma\leq \gamma_1$, then, for all $B_2>0$, 
\begin{equation*}\label{range of x that g1f-g2h>0}
g(t)f(t,x)- h(t,x)> 0, \quad \text{ for all } t\geq T, \vert x\vert \leq 2jt.
\end{equation*}
\item[$(iii)$] There is $T^{0}>0$ such that, for all $B_2>0$,
\begin{equation*}\label{range of x that g1f<g2h}
g(t)f(t,x)- h(t,x)\le 0, \quad \text{ for all } t\geq T^0, \vert x\vert = 2kt.
\end{equation*}
\end{itemize}
\end{lemma}

\begin{proof} Since $f(t,\cdot)$ and $h(t,\cdot)$ are even functions, it is enough to deal with $x\ge 0$. Recalling that $h_0(x)=B_3e^{-k\vert x\vert}$, we have
$$
h(t,x)=\frac{B_3}{\sqrt{4\pi t}}\left(\int _{-\infty}^0 e^{-\frac{(x-y)^2}{4t}}e^{ky}dy+\int_0^{+\infty}e^{-\frac{(x-y)^2}{4t}}e^{-ky}dy\right),
$$
which can be recast, after some elementary computations,
\begin{equation}\label{expression of h}
h(t,x)=\frac{B_3}{\sqrt \pi}\left(e^{k^2t-kx}\int_{\frac{2kt-x}{2\sqrt{t}}}^{+\infty}e^{-w^2}dw+e^{k^2t+kx}\int_{\frac{2kt+x}{2\sqrt{t}}}^{+\infty}e^{-w^2}dw\right).
\end{equation}
Now, recalling that $\int_X^{+\infty}e^{-w^{2}}dw\leq \frac{e^{-X^{2}}}{2X}$ for any $X>0$, the above  expression implies that, for any $0\leq x\leq 2jt<2kt$, 
$$
h(t,x) \le\frac{B_3}{\sqrt \pi} e^{-\frac{x^2}{4t}}\left(\frac{\sqrt t}{2kt-x}+\frac{\sqrt t}{2kt+x}\right)\leq \frac{B_3}{\sqrt \pi}\frac{k}{k^{2}-j^{2}}t^{-\frac 12}e^{-\frac{x^2}{4t}},
$$
which proves $(i)$.

Recalling that $f_0(x)=B_2\mathbf 1_{(-1,1)}(x)$, we have
\begin{equation}\label{expression of f}
f(t,x)=\frac{B_2}{\sqrt{4\pi t}}\int_{-1}^{1}e^{-\frac{(x-y)^2}{4t}}dy=\frac{B_2}{\sqrt{\pi}}\int_{\frac{x-1}{2\sqrt{t}}}^{\frac{x+1}{2\sqrt{t}}}e^{-w^2}dw.
\end{equation}
Hence, from $g(t)\geq 1$, $B_3=\gamma B_2$, \eqref{expression of f} and $(i)$, we deduce that, for all $0\leq x\leq 2jt$ and $t\geq T$,
\begin{eqnarray}
g(t)f(t,x)-h(t,x)&\geq& \frac{B_2}{\sqrt \pi}\left( t^{-\frac{1}{2}}e^{-\frac{(x+1)^2}{4t}}-\gamma \frac{k}{k^{2}-j^{2}}t^{-\frac 12}e^{-\frac{x^2}{4t}}\right)\nonumber\\
&\geq & \frac{B_2}{\sqrt \pi}t^{-\frac 12}e^{-\frac{x^2}{4t}}\left(e^{-\frac{1}{4T}}e^{-j}-\gamma \frac{k}{k^{2}-j^{2}}\right),\label{bidule}
\end{eqnarray}
which is enough to prove $(ii)$.

From \eqref{expression of h}, we have $h(t,2kt)\geq \frac{B_3}{\sqrt \pi}e^{-k^2t}\int_0^{+\infty}e^{-w^2}dw=\frac{B_3}{\sqrt \pi}e^{-k^2t}\frac{\sqrt \pi}{2}$. Hence, from  $B_3=\gamma B_2$ and \eqref{expression of f}, we deduce that, for $t\geq T^{0}:=\frac{1}{2k}$,
\begin{eqnarray*}
g(t)f(t,2kt)-h(t,2kt)&\leq& \frac{B_2}{\sqrt \pi}\left(\Vert g\Vert _\infty  t^{-\frac{1}{2}}e^{-\frac{(2kt-1)^2}{4t}}-\gamma \frac{\sqrt \pi }{2}e^{-k^{2}t}\right)\\
&\leq & \frac{B_2}{\sqrt \pi}e^{-k^{2}t}\left(\Vert g\Vert _\infty e^{k}t^{-\frac 12}-\gamma \frac{\sqrt \pi}2\right),
\end{eqnarray*}
which is nonpositive, up to increasing $T^{0}$ if  necessary. The proof of $(iii)$ is complete.
\end{proof}

\begin{remark}
The statement $(ii)$ in Lemma \ref{lm: some estimates of f and h} guarantees that $U(t,x)=g(t)f(t,x)-h(t,x)$ is not a trivial sub-solution if $\gamma$ is small enough.
\end{remark}

\begin{lemma}\label{lm: estimate of partial t f}
There exists $C=C(k)>0$ such that 
\begin{equation}
\label{t-un-demi}
|f(t,x)|+|h(t,x)|\le 
Ct^{-\frac{1}{2}},\quad  \text{ for all } t>0, x\in \R,
\end{equation}
and
\begin{equation}
\label{t-trois-demi}
|\partial_tf(t,x)|+|\partial_th(t,x)|\le 
Ct^{-\frac{3}{2}},\quad  \text{ for all } t>0, x\in \R.
\end{equation}
\end{lemma}

\begin{proof} This proof is very classical. Denoting $G(t,x):=\frac{1}{\sqrt{4\pi t}}e^{-\frac{x^2}{4t}}$, we have $f(t,x)=(G(t,\cdot)*f_0)(x)$, and thus
$$
\vert f(t,x)\vert \leq \Vert  G(t,\cdot)\Vert _\infty \Vert f_0\Vert _1\leq  C_G t^{-\frac 1 2}\Vert f_0\Vert _1,
$$
with some $C_G>0$. Also, we have $\partial _t f(t,x)=(\partial _t G(t,\cdot)*f_0)(x)$, and thus
$$
\vert \partial _t f(t,x)\vert \leq \Vert \partial _t G(t,\cdot)\Vert _\infty \Vert f_0\Vert _1\leq  C'_G t^{-\frac 32}\Vert f_0\Vert _1,
$$
with some $C'_G>0$. Note that, $\Vert f_0\Vert _1 =2B_2$, which implies $C$ is independent on $B_2<1$ in \eqref{t-un-demi} and \eqref{t-trois-demi}.

Since $h_0\in L^{1}(\R)$, similar estimates hold for $h(t,x)$ and $\partial _t h(t,x)$, and $C=C(k)$ since $\Vert h_0\Vert_1=\frac{2B_3}{k}\leq \frac 2 k$.
\end{proof}

We are now in the position to complete the construction of the sub-solution $(U,V)$ in the form \eqref{definition of sub sol system}.

\begin{proposition}[Sub-solutions]\label{prop:sub}  Let $0<\delta<\theta <\frac 12$ be given. Let us  fix $k>0$, and set $B_3=\gamma B_2$ with $0<\gamma<1$.

Then there exists $T^{*}>0$ such that, for all $0<B_2<1$ and $\zeta _0>0$, $(U,V)$ is a sub-solution in the domain $\Omega_2(T^{*})$ as  defined in \eqref{def:omega-zero}.
\end{proposition}

\begin{proof} Since $U(t,\cdot)$ and $V(t,\cdot)$ are even functions, it is enough to deal with $x\geq 0$. In other words, we work for $t\geq T$ (with $T>0$ to be selected) and $0\leq x<c_2t$, with $c_v<c_2<c_v^{**}$.  For simplicity of notations, we shall use the shortcuts  $\xi_{\pm}:=\pm x-c^{**}_vt-\zeta_0$. Since $\xi_-\leq -c_v^{**}t$ and $\xi_+\leq -(c_v^{**}-c_2)t$, it follows from $V_2'<0$ and \eqref{estimate of V at infinity} that  there exist $C_->0$ and $C_+>0$ such that, for $T>0$ large enough,
\begin{equation}\label{estimate of 1-V}
1-V_2(\xi_-)\leq C_-e^{-\lambda_2c^{**}_vt} \quad \text{ and }\quad  1-V_2(\xi_+)\leq C_+e^{-\lambda_2(c^{**}_v-c_2)t},\quad  \text{for all}\ (t,x)\in \Omega_2^+(T),
\end{equation}
where $\Omega_2^+(T):=\Omega_2(T)\cap (T,\infty)\times [0,\infty)$. Moreover, up to enlarging $T>0$ if necessary, there exists $0<\rho<\frac 1 3$ such that 
\begin{equation}
\label{rho}
0<1-V_2(\xi_\pm)\leq \rho, \quad  \text{for all}\ (t,x)\in \Omega_2^+(T).
\end{equation}

Some straightforward computations combined with  \eqref{def-f-h} yield
\begin{eqnarray*}
N_1[U,V]&=&g'f-(gf-h)(2-V_2(\xi_+)-V_2(\xi_-)-t^{-(1+\theta)})\\
&\leq &-(1+t)^{-(1+\delta)}gf+gft^{-(1+\theta)}+h\left(2-V_2(\xi_+)-V_2(\xi_-)-t^{-(1+\theta)}\right),
\end{eqnarray*}
since $2-V_2(\xi_+)-V_2(\xi_-)>0$.  Thus, it follows from \eqref{estimate of 1-V} that
$$
N_1[U,V]\leq gf\left(  -(1+t)^{-(1+\delta)}+t^{-(1+\theta)}\right)+h\left(C_-e^{-\lambda_2c^{**}_vt}+C_+e^{-\lambda_2(c^{**}_v-c_2)t}-t^{-(1+\theta)}\right).
$$
Since  $\delta<\theta$, it follows that, for $T>0$ large enough,  $N_1[U,V]\le 0$ in $\Omega_2^{+}(T)$.

Next, some straightforward computations combined with  \eqref{def-f-h} and \eqref{single-V0} yield
\begin{eqnarray*}
N_2[\bar{U},V]&=& r_2V_2(\xi_+)(1-V_2(\xi_+))+r_2V_2(\xi_-)(1-V_2(\xi_-))
+(1+t)^{-(1+\delta)}gf\\&&+(d-1)(g\partial_t f-\partial_t h)
-(1+\theta)t^{-(2+\theta)}\\
&&-r(V_2(\xi_+)+V_2(\xi_-)-1-gf+h+t^{-(1+\theta)})(2-V_2(\xi_+)-V_2(\xi_-)-t^{-(1+\theta)})\\
&=&I_1+\cdots +I_5,
\end{eqnarray*}
where
\begin{eqnarray*}
I_1&:=&rgf\left(2-V_2(\xi_+)-V_2(\xi_-)+\frac 1 r (1+t)^{-(1+\delta)}\right),\\
I_2&:=&(1-V_2(\xi_-))\left((r_2-r)V_2(\xi_-)+r\left(2-2V_2(\xi_+)-t^{-(1+\theta)}-h\right)\right),\\
I_3&:=&(1-V_2(\xi_+))\left((r_2-r)V_2(\xi_+)-rh\right),\\
I_4&:=&(d-1)(g\partial_t f-\partial_t h),\\
I_5&:=&rt^{-(1+\theta)}\left(V_2(\xi_-)+2V_2(\xi_+)-2+t^{-(1+\theta)}-gf+h-\frac{1+\theta}{r}t^{-1}\right).
\end{eqnarray*}
Since $0<V_2<1$, we have $I_1\geq 0$. From $r_2>r$, \eqref{t-un-demi} and \eqref{rho}, we have
$$
I_2\geq (1-V_2(\xi_-))\left((r_2-r)(1-\rho)-rt^{-(1+\theta)}-rCt^{-\frac 12}\right)\geq 0,
$$
up to enlarging $T>0$ if  necessary. Similarly, we obtain $I_3\geq 0$. Last, from \eqref{rho} and Lemma \ref{lm: estimate of partial t f}, we obtain
$$
I_4+I_5\ge r t^{-(1+\theta)}\left(1-3\rho-\Vert g\Vert _\infty C t^{-\frac 12}-\frac{1+\theta}{r}t^{-1}\right)-C\vert d-1\vert (\Vert g\Vert _\infty +1)t^{-\frac 32}.
$$ 
Since $\theta<\frac 12$ and $0<\rho <\frac 13$, we have $I_4+I_5\geq 0$ up to enlarging $T>0$ if  necessary. As a result, for $T>0$ large enough,  $N_2[U,V]\geq 0$ in $\Omega_2^{+}(T)$, and the proof of Proposition \ref{prop:sub} is complete.
\end{proof}

Note that, time $T^*$ in Proposition \ref{prop:sub} is independent on $0<B_2<1$ and $\zeta _0>0$, which leaves \lq\lq some room'' to reduce $B_2$ and to enlarge $\zeta _0$ so that the \lq\lq initial order'' is suitable for the comparison principle to be applicable. We shall also need the suitable \lq\lq order on the boundary of the domain'', which will be obtained by choosing an appropriate $k$ and Lemma \ref{lm: some estimates of f and h} $(iii)$. More precisely, the following holds.

\begin{proposition}[Second estimate on $(u,v)$]\label{prop:first}
Let $0<\delta<\theta <\frac 12$ be given. Let us  fix $k:=\frac{c_2}{2}>0$, and set $B_3=\gamma B_2$ with $0<\gamma<1$. 

Then there exist $T^{**}>0$, $0<B_2<1$ and $\zeta _0>0$ such that
$$
U(t,x)\leq u(t,x)\quad \text{ and }\quad  v(t,x)\le V(t,x), \quad \text{ for all }\; t\geq T^{**}, \vert x\vert \leq c_2t,
$$
where $(U,V)$ is given by \eqref{definition of sub sol system}.
\end{proposition}

\begin{proof} We aim at applying the comparison principle in $\Omega_2(T)$, as defined in \eqref{def:omega-zero}, with a well-chosen $T>0$. From Proposition \ref{prop:sub}, for any $T\geq T^{*}$, we are equipped with a sub-solution $(U,V)$ for which $0<B_2<1$ and $\zeta_0>0$ are arbitrary.

We now focus on  $\vert x\vert =c_2t$, $t\geq T$, that is, with a slight abuse of language, the boundary of $\Omega_2(T)$. We now set $T^{**}:=\max(T^{*},T^{0})$, where $T^0>0$ is provided by Lemma \ref{lm: some estimates of f and h} $(iii)$. In particular this implies (recall the choice $k=\frac{c_2}{2}$) that
$$
U(t,\pm c_2 t) \leq 0\leq u(t,\pm c_2 t), \quad \text{for all}\ t\geq T^{**}.
$$
Next,  recalling that $c_v<c_2<c_v^{**}$, it follows from  Proposition \ref{prop:estimate on leading edge} that $v(t,\pm c_2t)\to 0$ as $t\to\infty$. On the other hand, for any $t\geq T^{**}$,
$$
V(t,\pm c_2 t)\geq V_2(-(c_v^{**}-c_2)t)+V_2(-(c_2+c_v^{**})t)-1.
$$
As a result, up to enlarging $T^{**}$ if  necessary, one has
$$
v(t,\pm c_2 t)\leq V(t\pm c_2 t), \quad \text{for all } t\geq T^{**}.
$$

Last we focus on the initial data, namely $t=T^{**}$ and $\vert x\vert \leq c_2 T^{**}$. From \eqref{sandwich-0-1}, there exists $\ep>0$ such that
$$
\min\left(\inf_{|x|\le c_2T^{**}}u(T^{**},x),\inf_{|x|\le c_2T^{**}}1-v(T^{**},x)\right)\geq \ep>0.
$$
We now select $0<B_2 <\frac{\ep}{2\Vert g\Vert _\infty}$. From this choice, we have
$$
U(T^{**},x)\leq  \Vert g \Vert _\infty B_2\leq \frac \ep 2 \leq u(T^{**},x),\quad \text{ for all } \vert x\vert \leq c_2 T^{**},
$$
and, for all $\vert x\vert \leq c_2 T^{**}$,
\begin{equation*}\label{11}
V(T^{**},x)\ge 2V_{2}(-\zeta_0)-1-\Vert g\Vert _\infty B_2\ge 2V_{2}(-\zeta_0)-1-\frac \ep 2\geq 1-\ep,
\end{equation*}
by selecting $\zeta_0>0$ large enough, which implies
$$
v(T^{**},x)\leq V(T^{**},x), \quad  \text{ for all } \vert x\vert \leq c_2 T^{**}.
$$

As a consequence, the comparison principle can be applied in $\Omega_2(T^{**})$, which concludes the proof of Proposition \ref{prop:first}.
\end{proof}

\subsection{Proof of Theorem \ref{th: bump}}

It remains to prove the bump phenomenon which, as explained in Section \ref{s:results}, is reserved to the critical competition case under consideration. Let $0<\ep<c_u$ be given and let us prove \eqref{profile  u the solution dr>1 small x} and \eqref{profile of v the solution dr>1 small x}. Let us set $T^{**}>0$ such that both Proposition \ref{prop:first} and Proposition \ref{prop: boundary condition super sol} apply. In the sequel, we always consider $t\geq T^{**}$ and $0\leq x\leq \ep t$.

In particular, one has
$$
g(t)f(t,x)-h(t,x)=U(t,x)\leq u(t,x)\leq \tilde U(t,x)=t^{-(k^*-\frac{1}{2})}(1-e^{-\tau t})s(t,x).
$$
This estimate and \eqref{bidule} (with $j=\frac{\ep}{2}$)  yield the lower estimate in \eqref{profile  u the solution dr>1 small x}. On the other hand, Lemma \ref{lm: some estimates of f and h} $(i)$ (with $j=\frac{\ep}{2}$)  provides an upper bound of the form $t^{-\frac 12}e^{-\frac{x^2}{4t}}$ in the case  $d\leq 1$ (since then $\partial _t s= s_{xx}$) and of the form $t^{-\frac 12}e^{-\frac{x^2}{4dt}}$ in the case  $d\geq 1$  (since then $\partial _t s= d s _{xx}$). This gives the upper estimate in  \eqref{profile  u the solution dr>1 small x}.

Similarly, one obtains 
$$
U(t,x)-t^{-(1+\theta)}\leq 1-V(t,x)\leq 1-v(t,x)\leq 1-\tilde V(t,x)\leq 2- V_1(-c_v^{*}t)-V_1(-(c_v^{*}-\ep)t)+\tilde U(t,x),
$$
which gives the lower estimate in \eqref{profile of v the solution dr>1 small x}. On the other hand, we deduce from  \eqref{estimate of V at infinity-un} that there exist $C_->0$ and $C_+>0$ such that, for all $0\leq x\leq \ep t$,
\begin{eqnarray*}
1-\tilde V(t,x)&\leq&  C_-e^{-\lambda_1 c_v^{*}t}+C_+e^{-\lambda_1(c_v^{*}-\ep)t}+\tilde U(t,x)\\
&\leq & Ct^{-k^*}e^{-\frac{x^2}{4 d^* t}}+\tilde U(t,x),
\end{eqnarray*}
with some $C>0$ provided $\ep>0$ is chosen sufficiently small. This gives the upper estimate in \eqref{profile of v the solution dr>1 small x} and concludes the proof of Theorem \ref{th: bump}. \qed

\appendix 

\section{A result on the strong-weak competition system}

In this Appendix, we consider the strong-weak competition system ($0<a<1<b$)
\begin{equation}\label{strong weak competition system}
\left\{
\begin{aligned}
&\partial_t u=u_{xx}+u(1-u-av),\\
&\partial_t v=dv_{xx}+rv(1-v-bu),\\
\end{aligned}
\right.
\end{equation}
supplemented with an initial datum  $(u_0,v_0)$ satisfying \eqref{initial data}, for which we need a technical result, namely Lemma \ref{lm: decay estimamte u and v}, which is inspired by \cite[Lemma 2.8]{Peng Wu Zhou} and  the forthcoming work \cite{Xiao-Zhou}.

When $c_v=2\sqrt{rd}>c_u=2$, as proved in \cite{Girardin Lam}, the spreading properties are rather subtle: the rapid competitor $v$ invades first at speed $c_v$ and is then replaced by the strong competitor $u$ at a speed $\mathscr{C}$  which can take two different values. To make this clear, we quote the following from \cite{Girardin Lam} to which we refer for more details. First, the strong-weak system admits a {\it minimal} monotone traveling wave solution $(c_{LLW},U,V)$ with speed $2\sqrt{1-a}\leq c_{LLW}\leq 2$, defined as
\begin{equation*}\label{strong weak tw solution}
\left\{
\begin{aligned}
&U''+c_{LLW}U'+U(1-U-aV)=0,\\
&dV''+c_{LLW}V'+rV(1-V-bU)=0,\\
&(U,V)(-\infty)=(1,0),\ (U,V)(\infty)=(0,1),\\
&U'<0,\ V'>0.
\end{aligned}
\right.
\end{equation*} 
Next define the decreasing function $f:[2\sqrt{1-a},+\infty)\to (2\sqrt{a},2(\sqrt{1-a}+\sqrt a)]$  
\begin{equation*}
f(c):=c-\sqrt{c^2-4(1-a)}+2\sqrt{a}\ \ \mbox{ so that }\ \ f^{-1}(\tilde c):=\frac{\tilde c}{2}-\sqrt{a}+\frac{2(1-a)}{\tilde c-2\sqrt{a}}.
\end{equation*}
 If $2\sqrt{rd}\in(2,f(c_{LLW}))$, then define
\begin{equation*}
c_{nlp}:=f^{-1}(2\sqrt{rd})=\sqrt{rd}-\sqrt{a}+\frac{1-a}{\sqrt{rd}-\sqrt{a}} \in (c_{LLW},2).
\end{equation*}
Then we can precise
\begin{equation*}\label{definition of bar c}
\mathscr{C}=\left\{
\begin{aligned}
 c_{nlp} & \quad \text{ if }\; 2<c_v<f(c_{LLW}), \quad  \text{(acceleration phenomenon)},\\
 c_{LLW} &\quad  \text{ if }\; c_v\geq f(c_{LLW}).
\end{aligned}
\right.
\end{equation*}

We now state the result used in the proof of Proposition \ref{prop: boundary condition super sol}.

\begin{lemma}\label{lm: decay estimamte u and v} Assume $dr>1$ (i.e. $c_v>c_u$). Let $(u,v)=(u,v)(t,x)$ be the solution of the strong-weak competition system \eqref{strong weak competition system} starting from an initial datum $(u_0,v_0)=(u_0,v_0)(x)$ satisfying  \eqref{initial data}. Then the following holds.
\begin{itemize}
  \item[$(i)$] For any $c>\mathscr{C}$, there exist $C_1>0$, $\nu_1>0$, $T_1>0$ such that
 $$\sup_{\vert x\vert\geq c t} u(t,x)\leq C_1e^{-\nu_1 t},\quad \mbox{for all } t\geq T_1.$$
  \item[$(ii)$] For any $c_1$ and $c_2$ with $\mathscr{C}<c_1<c_2<c_v$, there exist $C_2>0$, $\nu_2>0$, $T_2>0$ such that
$$\sup_{c_1 t\leq \vert x\vert\leq c_2 t} v(t,x)\geq 1-C_2 e^{-\nu_2 t},\quad \mbox{for all } t\geq T_2.$$
\end{itemize}
\end{lemma}

\begin{proof} Let us briefly start with $(i)$. If $c>c_u$ then the conclusion is clear by the same argument as in Proposition \ref{prop:estimate on leading edge}. Since $c_v>c_u$, it thus suffices to consider the case $\mathscr{C}<c<c_v$. In this case, the conclusion is  already included in  \cite[Proposition~1.5]{Girardin Lam} and the proof of \cite[Section 3.2.3, Theorem 1.1]{Girardin Lam}. We do not present the full details but only emphasize that a key tool is, for any small $\delta>0$, the {\it minimal} monotone traveling wave of the perturbed system
\begin{equation*}\label{perturbed tw solution}
\left\{
\begin{aligned}
&U''+cU'+U(1+\delta-U-aV)=0,\\
&dV''+cV'+rV(1-2\delta-V-bU)=0,\\
&(U,V)(-\infty)=(1+\delta,0),\ (U,V)(\infty)=(0,1-2\delta),\\
&U'<0,\ V'>0.
\end{aligned}
\right.
\end{equation*}

\medskip

Let us now turn to $(ii)$ (which is the estimate we need in the proof of Proposition \ref{prop: boundary condition super sol}) for which the above perturbation argument seems unapplicable.  Let $\mathscr{C}<c_1<c_2<c_v$ be given. We only deal with $x\geq 0$. From \cite[Theorem 1.1]{Girardin Lam} we know 
\begin{equation*}
    \lim_{t\to\infty}\sup_{c_1 t\leq  x\leq c_2t}\Big(u(t,x)+|1-v(t,x)|\Big)=0.
    \end{equation*}
From this and $(i)$, one can choose $\varepsilon>0$ small enough and $T_0\gg1$ such that
$$0<u(t,x)\leq C_1 e^{-\nu_1 t},\quad v(t,x)> 1-\varepsilon,\quad\mbox{for all } t\geq  T_0,x\in[c_1 t, c_2 t].$$
From the $v$-equation in \eqref{strong weak competition system}, we have
\begin{equation}\label{v-ineq}
v_t\geq  d v_{xx}+r (1-\varepsilon)(1-v)-rbC_1v e^{-\nu_1 t},\quad\mbox{for all } t\geq T_0,x\in[c_1 t, c_2 t].
\end{equation}
Defining
\begin{equation}\label{def-tilde-v}
\tilde{v}(t,x):=v(t,x+\tilde{c}t),\quad \tilde{c}:=\frac{c_1+c_2}{2},
\end{equation}
it follows  from \eqref{v-ineq} that
$$\tilde{v}_t\geq d \tilde{v}_{xx}+\tilde{c}\tilde{v}_x+r (1-\varepsilon)(1-\tilde{v})-rbC_1\tilde{v} e^{-\nu_1 t},\quad\mbox{for all } t\geq T_0,x\in[-c_3t, c_3t],$$
where $c_3:=\frac{c_2-c_1}2$. 

To estimate $\tilde{v}$, for any $T\geq T_0$, we define
$$
\alpha(t):=1+\frac{bC_1}{1-\varepsilon}e^{-\nu_1(t+T)}, \quad\text{for all } t\geq 0.
$$
Up to enlarging $T>0$ if necessary, we may assume $\alpha(0)<\frac 1{1-\ep}$. Now, let us first consider the auxiliary problem
\begin{equation}\label{phi-sys}
\begin{cases}
\phi_t=d\phi_{xx}+\tilde{c}\phi_x+r (1-\varepsilon)[1-\alpha(t)\phi], &  t>0,\ -c_3T<x<c_3T,\\
\phi(t,\pm c_3{T})=1-\varepsilon, & t>0,\\
\phi(0,x)=1-\varepsilon, & -c_3{T}\leq x\leq c_3{T}.
\end{cases}
\end{equation}
Letting
$$\Phi(t,x):=e^{Q(t)}[\phi(t,x)-1+\varepsilon],\quad Q(t):=r(1-\varepsilon)t-\frac{rbC_1}{\nu_1}e^{-\nu_1(t+T)},$$
so that $Q'(t)=r(1-\varepsilon)\alpha(t)$, it follows from \eqref{phi-sys} that
\begin{equation}\label{Phi-sys}
\begin{cases}
\Phi_t=d\Phi_{xx}+\tilde{c}\Phi_x+r (1-\varepsilon)e^{Q(t)}(1-(1-\varepsilon)\alpha(t)), & t>0,\ -c_3T<x<c_3T,\\
\Phi(t,\pm c_3{T})=0, & t> 0,\\
\Phi(0,x)=0, & -c_3{T}\leq x\leq c_3{T}.
\end{cases}
\end{equation}
Up to a rescaling, we may assume $d=1$ so that \eqref{Phi-sys} is very comparable to \cite[problem (3.12)]{Kaneko Matsuzawa} on which we now rely. Denoting $G_1(t,x,z)$ the Green function of \cite[page 53]{Kaneko Matsuzawa} (with obvious changes of constants), we obtain the analogous of \cite[(3.14)]{Kaneko Matsuzawa}, namely
$$
\Phi(t,x)\ge r(1-\varepsilon)\int_0^te^{Q(s)}(1-(1-\varepsilon)\alpha(s))\left(\int_{-c_3T}^{c_3T}G_1(t-s,x,z)dz\right)ds,
$$
for all $t>0$, $ -c_3 T<x<c_3 T$. Next, for any small $0<\delta<1$, we define 
$$
D_{\delta}:=\left\{(t,x)\in \R^2 : 0<t <  \delta^2c_3 {T},  \vert x\vert <(1-\delta)c_3 T\right\}.
$$
From the same process used in \cite[pages 54-55]{Kaneko Matsuzawa}, there exist $C_3, C_4>0$ such that the following lower estimate holds
$$
\Phi(t,x)\ge r(1-\varepsilon)(1-(1-\varepsilon)\alpha(0))(1-C_3e^{-C_4T})\int_0^te^{Q(s)}ds, \quad\text{for all } (t,x)\in D_\delta,
$$
resulting in
\begin{equation}\label{inequality of phi}
\phi(t,x)\ge \Psi(t) (1-C_3e^{-C_4T})(1-(1-\varepsilon)\alpha(0))+1-\varepsilon,  \quad\text{for all } (t,x)\in D_\delta,
\end{equation}
where $\Psi(t):=r(1-\varepsilon)e^{-Q(t)}\int_0^te^{Q(s)}ds$. Denoting $K=\frac{rbC_1}{\nu_1}$, we have
\begin{equation*}
\begin{aligned}
\Psi(t)&\ge r(1-\varepsilon)e^{-Q(t)}\int_0^te^{r(1-\varepsilon)s}e^{-Ke^{-\nu_1T}}ds\\
&= e^{-r(1-\varepsilon)t}e^{Ke^{-\nu_1(t+T)}}e^{-Ke^{-\nu_1T}}\int_0^tr(1-\varepsilon)e^{r(1-\varepsilon)s}ds\\
&= e^{Ke^{-\nu_1T}(e^{-\nu_1t}-1)}(1-e^{-r(1-\varepsilon)t}).
\end{aligned}
\end{equation*}
Inserting this into \eqref{inequality of phi} and using  $e^y\ge 1+y$ for all $y\in\R$, we have, for all $(t,x)\in D_\delta$,
\begin{equation*}
\begin{aligned}
\phi(t,x)&\ge e^{Ke^{-\nu_1T}(e^{-\nu_1t}-1)}(1-e^{-r(1-\varepsilon)t})(1-C_3e^{-C_4T})(1-(1-\varepsilon)\alpha(0))+1-\varepsilon\\
&\ge (1-Ke^{-\nu_1T}(1-e^{-\nu_1t}))(1-e^{-r(1-\varepsilon)t})(1-C_3e^{-C_4T})(1-(1-\varepsilon)\alpha(0))+1-\varepsilon\\
&\ge (1-Ke^{-\nu_1T})(1-e^{-r(1-\varepsilon)t})(1-C_3e^{-C_4T})(1-(1-\varepsilon)\alpha(0))+1-\varepsilon.
\end{aligned}
\end{equation*}
Letting
$$
I_1:=1-Ke^{-\nu_1T}, \quad I_2(t):=1-e^{-r(1-\varepsilon)t}, \quad I_3:=1-C_3e^{-C_4T},
$$
we get
$$
\phi(t,x)\ge I_1I_2(t)I_3+(1-\varepsilon)(1-I_1I_2(t)I_3\alpha(0)),\quad\text{for all } (t,x)\in D_\delta. 
$$
Now observe that  $I_1I_2(t)I_3 \leq I_1 I_2(\delta ^2 c_3 T) I_3$. Furthermore some straightforward computations show that, if
\begin{equation}
\label{delta}
r(1-\ep)\delta ^2c_3<\nu _1,
\end{equation}
then $I_1 I_2(\delta ^2 c_3 T) I_3\alpha(0)\leq 1$ up to enlarging $T>0$ if necessary. As a result, for all $(t,x) \in D_\delta$, 
$$
\varphi(t,x) \geq I_1I_2(t)I_3\geq  1-K_1e^{-\nu_1T}-K_2e^{-r(1-\varepsilon)t},
$$
with some $K_1, K_2>0$. The last inequality holds since we can always choose $\nu_1<C_4$. As a conclusion, we have
\begin{equation}\label{phi-estiamte}
\phi(t,x)\geq 1-K_1e^{-\nu_1 T}-K_2e^{-r(1-\varepsilon)t},\quad\text{for all } (t,x)\in D_{\delta},
\end{equation}
provided that $\delta>0$ is sufficiently small for \eqref{delta} to hold and  $T>0$ is sufficiently large. 

In particular \eqref{phi-estiamte} implies that, for all $\vert x\vert \leq (1-\delta)c_3T$,
\begin{equation}\label{phi-estimate2}
\phi\left(\delta^2 c_3 {T},x\right) \geq  1-K_1e^{-\nu_1 {T}}- K_2 e^{-r(1-\varepsilon) \delta^2 c_3{T}}
                               \geq 1-(K_1+K_2) e^{-r(1-\varepsilon) \delta^2 c_3 {T}},
                                                                 \end{equation}
in virtue of \eqref{delta}.  On the other hand, we know from the comparison principle that $\tilde{v}(t+T,x)\geq \phi(t,x)$ for $t\geq0$ and $|x|\leq c_3 T$, which
together with \eqref{phi-estimate2} implies that
$$\tilde{v}\left(\delta^2c_3 T+T,x\right)\geq1-(K_1+K_2) e^{-r(1-\varepsilon) \delta^2 c_3 {T}},\quad \mbox{for all } |x|\leq (1-\delta)c_3 T.$$
We further take
$t=(\delta^2c_3+1)T$,
which yields 
$$\tilde{v}(t,x)\geq 1- C_2e^{-\nu_2t},\quad \mbox{for all large $t$ and $|x|\leq \frac{(1-\delta)c_3}{1+c_3\delta^2}t$},$$
where $C_2:=K_1+K_2$ and $\nu_2:=\frac{r(1-\varepsilon)\delta^2c_3}{1+c_3\delta^2}>0$. Recalling that $\tilde{v}(t,x)=v(t,x+\tilde{c}t)$ with $\tilde{c}=\frac{c_1+c_2}{2}$, that $c_3=\frac{c_2-c_1}{2}$ and since $\delta>0$ can be chosen arbitrarily small, the above estimate completes the proof of $(ii)$. 
\end{proof}

\bigskip
 
 \noindent{\bf Acknowledgement.} M. Alfaro is supported by the ANR project DEEV ANR-20-CE40-0011-01. D. Xiao is supported by the LabEx {\it Solutions Numériques, Matérielles et Modélisation pour l’Environnement et le Vivant} (NUMEV) of the University of Montpellier.

\end{document}